\documentclass[12pt,twoside]{article}

\usepackage[letterpaper,left=3.3cm, right=3.3cm, top=3.5cm,
bottom=3.5cm]{geometry}
\usepackage{amsfonts,amsmath,amsthm,amssymb,hyperref,xypic,tikz}
\usetikzlibrary{decorations.markings} \usetikzlibrary{arrows}
\usepackage[scaled]{helvet} 
\usepackage{courier} 
\usepackage[euler-hat-accent]{eulervm} 
\normalfont \usepackage[T1]{fontenc} \xyoption{all}

\newtheorem{theorem}{Theorem}[section]
\newtheorem{corollary}[theorem]{Corollary}
\newtheorem{lemma}[theorem]{Lemma}
\newtheorem{prop}[theorem]{Proposition}

\theoremstyle{definition} \newtheorem{remark}[theorem]{Remark}
\newtheorem{defn}[theorem]{Definition}
\newtheorem{example}[theorem]{Example}
\numberwithin{equation}{section}

\newcommand{\nc}{\newcommand}\nc{\br}{\overline} \nc{\FH}{\mathcal H}
\nc{\CC}{\mathbb C}\nc{\DD}{\mathcal D} \nc{\Cc}{\mathcal
  C}\nc{\JJ}{\mathbb J}\nc{\II}{\mathbb I}\nc{\I}{\mathcal{I}}
\nc{\PP}{\mathbb P} \nc{\KK}{\mathbb K} \nc{\RR}{\mathbb R}
\nc{\LL}{\mathcal L} \nc{\Ll}{\ell} \nc{\NN}{\mathbb N}
\nc{\ZZ}{\mathbb Z} \nc {\HH}{\mathbb H} \nc {\OO}{\mathcal{O}}
\nc{\lra}{\longrightarrow} \nc{\bdot}{\bullet} \nc{\w}{\omega}
\nc{\dd}{\mathcal{D}}\nc{\im}{\mathrm{Im}}\nc{\Nij}{\mathrm{Nij}}\nc{\Jac}{\mathrm{Jac}}

\nc{\sgn}{\mathrm{sgn}}
\newcommand{\Tr}{\mathrm{Tr}}
\newcommand{\Hom}{\mathrm{Hom}}

\newcommand{\delbar}{\overline{\partial}}
\nc{\MM}{\mathcal{M}}\nc{\Har}{\mathcal{H}}
\nc{\zed}{\mathcal{Z}} \nc{\bb}[1]{\mathbb{#1}}
\nc{\image}{\mathrm{Im}\ } 
 \nc{\ba}{\overline}
\nc{\del}{\partial} \nc{\AAA}{\mathcal{A}}
\nc{\de}{\delta}\nc{\debar}{\overline{\delta}}
\nc{\FF}{\mathcal{F}}\nc{\GG}{\mathcal{G}}\nc{\EE}{\mathcal{E}}
\nc{\Ja}{e^{\tfrac{\pi}{2}\JJ_1}} \nc{\Jb}{e^{\tfrac{\pi}{2}\JJ_2}}
\nc{\eps}{\epsilon}\nc{\id}{\mathrm{id}}\nc{\Dir}{\mathrm{Dir}}\nc{\SO}{\mathrm{SO}}
\nc{\IPS}[1]{#1}\nc{\veps}{\varepsilon}\nc{\Cour}[1]{[#1]}\nc{\Diff}{\mathrm{Diff}}
\nc{\ad}{\mathrm{ad}} \nc{\UU}{\mathcal{U}} \nc{\Aa}{\mathcal
  A}\nc{\Bb}{\mathcal{B}}\nc{\Pic}{\mathrm{Pic}}\nc{\type}{\mathrm{type}}
\nc{\TT}{\mathbb{T}}\nc{\T}{\mathcal{T}}
\nc{\cl}{\mathrm{cl}}\nc{\isom}{\cong}
\nc{\Z}{\mathcal{Z}}\nc{\Tot}{\mathrm{Tot}}\nc{\Dol}{\mathrm{Dol}}\nc{\hol}{\mathrm{hol}}
\nc{\E}{\mathbb{E}}\nc{\e}{\mathbf{e}} \nc{\Gg}{\mathfrak{g}}
\nc{\Aut}{\mathrm{Aut}}\nc{\Der}{\mathrm{Der}}\nc{\ol}{\overline}\nc{\til}{\widetilde}
  \newcommand{\gcss}{generalized complex
  structures} \newcommand{\gk}{generalized K\"ahler}
\newcommand{\gkm}{generalized K\"ahler manifold}
\newcommand{\gks}{generalized K\"ahler structure}

\newcommand{\nhood}{neighborhood} 
 \newcommand{\C}{\mathbb{C}}
\newcommand{\mc}[1]{\mathcal{#1}}
\newcommand{\cs}{\hspace{1pt}\#\hspace{1pt}}
\renewcommand{\tilde}[1]{\widetilde{#1}}

\title{\bf Blowing up generalized K\"ahler 4-manifolds} \author{Gil
  R. Cavalcanti\footnote{Utrecht University;
   \href{mailto:g.r.cavalcanti@uu.nl}{\nolinkurl{g.r.cavalcanti@uu.nl}}
   \newline\hspace*{2em}Supported by a Marie Curie grant from the European Research Council. 
  }\and Marco
  Gualtieri\footnote{University of Toronto; \href{mailto:mgualt@math.toronto.edu}{\nolinkurl{mgualt@math.toronto.edu}}
\newline\hspace*{2em}Supported by a NSERC Discovery grant and an Ontario ERA.}}
\date{} \usepackage{color} \definecolor{tocolor}{rgb}{.1,.1,.5}
\definecolor{urlcolor}{rgb}{.2,.2,.6}
\definecolor{linkcolor}{rgb}{.1,.1,.6}
\definecolor{citecolor}{rgb}{.6,.2,.1}
\hypersetup{backref=true, colorlinks=true, urlcolor=urlcolor,
  linkcolor=linkcolor, citecolor=citecolor}

\begingroup
\hypersetup{linkcolor=blue}

\endgroup

\begin{document}

\maketitle

\abstract{We show that the blow-up of a generalized K\"ahler
  4-manifold in a nondegenerate complex point admits a generalized
  K\"ahler metric.  As with the blow-up of complex surfaces, this
  metric may be chosen to coincide with the original outside a tubular
  neighbourhood of the exceptional divisor. To accomplish this, we
  develop a blow-up operation for bi-Hermitian manifolds.}

\tableofcontents
\pagebreak
\section{Introduction}\label{intro}
Let $(M,\JJ_+,\JJ_-)$ be a generalized K\"ahler 4-manifold such that
both generalized complex structures $\JJ_+,\JJ_-$ have even type,
meaning that they are equivalent to either a complex or symplectic
structure at every point.  In other words, their underlying real
Poisson structures $P_+, P_-$ have either rank $0$ (at complex points)
or $4$ (at symplectic points).  The structure $\JJ_\pm$ is equipped
with a canonical section $s_\pm$ of its anticanonical line bundle,
vanishing on the locus $D_\pm$ of complex points, where $P_\pm$ has
rank zero.  From~\cite{Gualtieri:2010fk}, it follows that the
symplectic leaves of $P_+$ and $P_-$ must be everywhere transverse, so
that $D_+, D_-$ are disjoint.
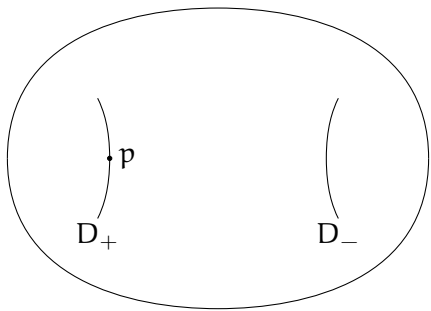
\begin{center}\begin{tikzpicture}[scale=.8]
    \draw (-3.5,0) .. controls (-3.5,2) and (-1.5,2.5) .. (0,2.5);
    \draw[xscale=-1] (-3.5,0) .. controls (-3.5,2) and (-1.5,2.5)
    .. (0,2.5); \draw[rotate=180] (-3.5,0) .. controls (-3.5,2) and
    (-1.5,2.5) .. (0,2.5); \draw[yscale=-1] (-3.5,0) .. controls
    (-3.5,2) and (-1.5,2.5) .. (0,2.5); \draw (-2,-1) .. controls
    (-1.9,-0.8) and (-1.8,-0.5) .. (-1.8,0) .. controls (-1.8,+0.5)
    and (-1.9,0.8) .. (-2,1); \filldraw [black] (-1.8,0) circle (1pt);
    \draw (-1.5,0) node {$p$}; \draw (-2,-1.3) node {$D_+$}; \draw
    (+2,-1) .. controls (+1.9,-0.8) and (+1.8,-0.5) .. (+1.8,0)
    .. controls (+1.8,+0.5) and (+1.9,0.8) .. (+2,1); \draw (2,-1.3)
    node {$D_-$};
  \end{tikzpicture}\end{center}
It was shown in \cite{Cavalcanti:2008uk} that in a neighbourhood of a
complex point $p\in D_+$ which is nondegenerate, in the sense of being
a nondegenerate zero of $s_+$, there are complex coordinates $(w,z)$
such that the generalized complex structure $\JJ_+$ is equivalent to
that defined by the differential form
\begin{equation}\label{standard}
  \rho_+ = w + dw\wedge dz.
\end{equation}
Note that $D_+ = w^{-1}(0)$, along which $\rho_+|_{D_+} = dw\wedge dz$
defines a complex structure, whereas for $w\neq 0$, we have $\rho_+ =
w\exp(B+i\omega)$, for $B+i\omega = d\log w \wedge dz$, defining a
symplectic form $\omega$ away from $D_+$, as required.

It was then shown~\cite[Theorem 3.3]{Cavalcanti:2008uk} that the
complex blow-up at $p$ using the coordinates $(z,w)$ inherits a
generalized complex structure.  We detail in Section~\ref{blowp} why
this structure is independent of the chosen coordinates.  Thus we
obtain a canonical blow-up $(\til M, \til\JJ_+)$ of $(M,\JJ_+)$ at
$p$, equipped with a generalized holomorphic map $\pi:\til M\lra M$
which is an isomorphism outside the exceptional divisor $E =
\pi^{-1}(p)$.  The complex locus $\til D_+$ of the blow-up is the
proper transform of $D_+$, and the exceptional divisor $E$ is a
2-sphere which intersects $\til D_+$ transversely at one point and is
Lagrangian with respect to $\omega$ elsewhere; this makes $E$ a
generalized complex brane~\cite{Cavalcanti:2008uk}.

\begin{center}\begin{tikzpicture}
    \draw (-2,-1) .. controls (-1.9,-0.8) and (-1.8,-0.5) .. (-1.8,0)
    .. controls (-1.8,+0.5) and (-1.9,0.8) .. (-2,1);
    \draw[densely dashed] (-3,-.05) .. controls (-2.7,.05) and
    (-2.3,.1) .. (-2,.1) .. controls (-1.7,.1) and (-1.3,0.05)
    .. (-1,-.05); \draw (-2.6,.3) node {$E$}; \draw (-2,-1.3) node
    {$\til D_+$};

    \draw[ decoration={markings, mark=at position 1 with
      {\arrow[scale=2]{>}}}, postaction={decorate}, shorten >=0.4pt]
    (0,0) -- (1.9,0); \draw (.9,.3) node {$\pi$};

    \draw (3,-1) .. controls (3.1,-0.8) and (3.2,-0.5) .. (3.2,0)
    .. controls (3.2,+0.5) and (3.1,0.8) .. (3,1); \filldraw [black]
    (3.2,0) circle (1pt); \draw (3.5,0) node {$p$}; \draw (3,-1.3)
    node {$D_+$};
  \end{tikzpicture}\end{center}
In Section~\ref{gkblowup}, we use the bi-Hermitian tools developed in
Section~\ref{biherm} to construct a \emph{degenerate} generalized
K\"ahler structure on the blow-up, in the sense that the metric
degenerates along the exceptional divisor $E$. Finally, in
Section~\ref{gkblowup2}, we use a deformation procedure detailed in
Section~\ref{flowsec} to obtain a positive-definite metric, defining a
generalized K\"ahler structure such that $\pi$ is an isomorphism away
from a tubular neighbourhood of the exceptional divisor $E$.  The
generalized complex structure $\JJ_-$ does not lift uniquely to the
blow-up, as there is no preferred choice of symplectic area for $E$;
this degree of freedom inherent in the generalized K\"ahler blow-up is
familiar from the usual K\"ahler blow-up operation.

\section{Generalized complex blow-up}\label{blowp}

Let $(w,z)$ be standard coordinates for $M=\CC^2$, and consider the
generalized complex structure $\JJ$ defined by the form $\rho_+$ given
in~\eqref{standard}.  This structure extends uniquely to a generalized
complex structure $\til\JJ$ the blow-up $\til M = [\CC^2:0]$ of the
plane in the origin, simply because the anticanonical section
$\sigma=w\del_w\wedge\del_z$ does.  That is, the line generated by
$\rho_+$ may be written
\[
\left<\rho_+\right>=e^\sigma\Omega^{2,0}(M),
\]
and in the two blow-up charts $(w_0,z_0)=(w/z,z)$ and $(w_1, z_1) =
(w, z/w)$, this pulls back to the line
$e^{\til\sigma}\Omega^{2,0}(\til M)$, where
\[
\til\sigma = w_0\del_{w_0}\wedge\del_{z_0} =
\del_{w_1}\wedge\del_{z_1}.
\]
Clearly, $\til\sigma$ drops rank along the proper transform of
$w^{-1}(0)$, namely $w_0^{-1}(0)$.

The above construction of $\til\JJ$ uses the complex structure defined
by $(w,z)$, but this complex structure is not determined canonically
by $\JJ$.  That is, there are automorphisms
$\Phi=(\varphi,B)\in\Diff(M)\ltimes\Omega^{2,\cl}(M,\RR)$ of $\JJ$ for
which $\varphi$ is not a holomorphic automorphism of $\CC^2$.  To show
that $\til\JJ$ is independent of the particular complex structure used
to perform the blow-up, we must show that any such automorphism
$\Phi\in\Aut(\JJ)$ with $\varphi(0)=0$ lifts to the blow-up
$[\CC^2:0]$.

\begin{theorem}\label{indep}
  Any automorphism of $\JJ$ on $M=\CC^2$ fixing the origin lifts to
  the blow-up $\til M$ of $M$ in the origin.
\end{theorem}
\begin{proof}
  Let $\Phi=(\varphi,B)\in\Aut(\JJ)$, meaning that
  \begin{equation}\label{autex}
    e^B\varphi^*(w + dw\wedge dz)=e^\lambda(w+dw\wedge dz),
  \end{equation}
  for some $\lambda\in C^\infty(M,\CC)$.  Also, assume $\varphi(0)=0$.
  Let $p:\til M\to M$ be the blow-down map.  We will show that
  $\varphi$ lifts to $\til\varphi\in\Diff(\til M)$ such that $p\circ
  \til\varphi = \varphi\circ p$, and then $(\til \varphi,
  p^*B)\in\Aut(\til\JJ)$ is the required lift of the automorphism.
  The lift $\til\varphi$ exists if and only if the functions $\til w =
  \varphi^*w$, $\til z = \varphi^*z$ are in the ideal generated by $w$
  and $z$ in $C^\infty(M,\CC)$.  By a theorem of
  Malgrange~\cite{MR0212575}, this is equivalent to the following
  constraints: $\til w(0)=0$, $\til z(0)=0$, and
  \begin{equation}\label{lift}
    \left.\frac{\del^{p+q} \til w}{\del^p\ol w\ \del^q \ol z}\right|_{(0,0)} = 0\ \ \ \text{and}\ \ \ \left.\frac{\del^{p+q} \til z}{\del^p\ol w\ \del^q \ol z}\right|_{(0,0)} = 0,\ \text{for all }p,q\in\NN.
  \end{equation}
  To verify~\eqref{lift}, we rewrite~\eqref{autex} as follows:
  \begin{equation}\label{autexp}
    \til w + d\til w\wedge d\til z = e^\lambda e^{-B}(w + dw \wedge dz) = e^\lambda(w + dw\wedge dz - w B),
  \end{equation}
  where the summand of degree four is omitted from the last term since
  it vanishes.  From this we immediately conclude that $\til w =
  e^\lambda w$, so that $\til w$ satisfies~\eqref{lift}. But then
  \begin{align*}
    d\til w\wedge d\til z &= d(e^\lambda w)\wedge d\til z\\
    &=e^\lambda(dw + wd\lambda)\wedge (\tfrac{\del \til z}{\del \ol w}
    d\ol w + \tfrac{\del \til z}{\del w} dw + \tfrac{\del \til z}{\del
      \ol z} d\ol z + \tfrac{\del \til z}{\del z} d z).
  \end{align*}
  By~\eqref{autexp}, this coincides with $e^\lambda(dw\wedge dz -
  wB)$, and equating $dw\wedge d\ol z$ components we obtain
  \[
  (1 + w\tfrac{\del \lambda}{\del w})\tfrac{\del\til z}{\del\ol z} - w
  \tfrac{\del \lambda}{\del \ol z} \tfrac{\del \til z}{\del w} =
  -wB_{w\ol z}.
  \]
  Solving for $\tfrac{\del\til z}{\del\ol z}$ we obtain, near $(0,0)$,
  \begin{equation}\label{firstcon}
    \frac{\del \til z}{\del\ol z} = \frac{w (\tfrac{\del \lambda}{\del \ol z} \tfrac{\del \til z}{\del w} -B_{w\ol z})}{1 + w\tfrac{\del \lambda}{\del w}}.
  \end{equation}
  Similarly, equating $dw\wedge d\ol w$ components yields, near
  $(0,0)$,
  \begin{equation}\label{seccon}
    \frac{\del \til z}{\del\ol w} = \frac{w (\tfrac{\del \lambda}{\del \ol w} \tfrac{\del \til z}{\del w} -B_{w\ol w})}{1 + w\tfrac{\del \lambda}{\del w}}.
  \end{equation}
  Finally,~\eqref{firstcon},~\eqref{seccon} imply that~\eqref{lift}
  holds for $\til z$, as required.
\end{proof}

\section{Bi-Hermitian approach}\label{biherm}
Our main tool for describing the geometry of the blow-up will be the
bi-Hermitian approach to generalized K\"ahler
geometry~\cite{Gualtieri:2010fk}, which we describe briefly here.
Since we are interested in a neighbourhood of a point, we may assume
that the torsion 3-form $H$ of our generalized K\"ahler structure is
cohomologically trivial.  Such a generalized K\"ahler structure
determines and is determined by a Riemannian metric $g$, a 2-form $b$,
and a pair of complex structures $I_+, I_-$ which are compatible with
$g$ and satisfy the condition
\begin{equation}\label{gkcon}
  \pm d^c_\pm\omega_\pm= db,
\end{equation}
where $\omega_\pm = g I_\pm$ are the usual Hermitian 2-forms and
$d^c_\pm = [d,I_\pm^*]$ are the real Dolbeault operators associated to
$I_\pm$.  The correspondence between the generalized K\"ahler pair
$\JJ_+,\JJ_-$ and the above bi-Hermitian data is as follows:
\begin{equation}
  \label{reconstruct}
  \JJ_\pm=\tfrac{1}{2}
  \begin{pmatrix}
    1 &  \\
    -b & 1
  \end{pmatrix}
  \begin{pmatrix}I_+\pm I_- & -(\omega_+^{-1}\mp\omega_-^{-1}) \\
    \omega_+\mp\omega_-&-(I^*_+\pm I^*_-)\end{pmatrix}
  \begin{pmatrix}
    1 &  \\
    b & 1
  \end{pmatrix}.
\end{equation}
It was observed in~\cite{MR2217300} that the bi-Hermitian condition
endows the complex structure $I_\pm$ with a holomorphic Poisson
structure $\sigma_\pm$ with real part
\begin{equation}\label{cmmt}
  Q=\mathrm{Re}(\sigma_+)=\mathrm{Re}(\sigma_-) = \tfrac{1}{8}[I_+,I_-]g^{-1}.
\end{equation}
Indeed, $\sigma_\pm$ derives from a pair of transverse holomorphic
Dirac structures as described in~\cite{Gualtieri:2010fk}, though we
shall not make use of this here.

Any pair of complex structures satisfies the following identity for
the commutator:
\begin{equation} \label{ident} [I_+,I_-]=(I_+-I_-)(I_- + I_+).
\end{equation}
Therefore, the zeros of $Q$ coincide with the loci where $I_+ = I_-$
or $I_+ = -I_-$.  From~\eqref{reconstruct}, we see that the real
Poisson structures $P_\pm$ underlying $\JJ_\pm$ are given by
\begin{equation}\label{realpois}
  P_\pm = -\tfrac{1}{2}(\omega_+^{-1} \mp \omega_-^{-1}) = \tfrac{1}{2}(I_+\mp I_-)g^{-1}.
\end{equation}
Therefore, we conclude that the zero locus of $Q$, and hence
$\sigma_\pm$, is the union of the zero loci for $P_+,P_-$, namely the
subsets $D_+, D_-$ discussed in section~\ref{intro}.

The holomorphic Poisson structure $(I_\pm, \sigma_\pm)$ provides an
economical means to describe the full generalized K\"ahler structure,
as observed in~\cite{Gualtieri:2007fe}.
\begin{theorem}[\cite{Gualtieri:2007fe}, Theorem 6.2]\label{brGK}
  Let $(I_0,\sigma_0)$ be a holomorphic Poisson structure with
  $\mathrm{Re}(\sigma_0)=Q$.  Any closed 2-form $F$ satisfying the
  equation
  \begin{equation}\label{brane}
    FI_0 + I_0^*F + FQF = 0
  \end{equation}
  defines an integrable complex structure $I_1 = I_0+ QF$, a symmetric
  tensor $g = -\tfrac{1}{2}F(I_0+I_1)$, and a 2-form $b =
  -\tfrac{1}{2}F(-I_0 + I_1)$ such that
  \[
  d^c_{I_0}\omega_{I_0} = -d^c_{I_1}\omega_{I_1} = db.
  \]
  If $g$ is positive-definite, then $(g,I_0,I_1)$ defines a
  bi-Hermitian structure satisfying~\eqref{gkcon}, and hence a
  generalized K\"ahler structure, where $\JJ_-$ is the symplectic
  structure $F$.
\end{theorem}

As is hinted at in Theorem~\ref{brGK}, in which $g$ need not be
positive-definite, it will be useful in studying the blowup for us to
relax the generalized K\"ahler condition, allowing degenerations of
the Riemannian metric while maintaining the remaining constraints.
\begin{defn}
  A \emph{degenerate bi-Hermitian} structure $(g,b,I_+, I_-)$ consists
  of a possibly degenerate tensor
  $g\in\Gamma^\infty(\mathrm{Sym}^2T^*)$, a 2-form $b\in\Omega^2$, and
  two integrable complex structures $I_+, I_-$, such that $gI_\pm +
  I_\pm^* g = 0$ and
  \begin{equation*}
    d^c_+ \omega_+ = -d^c_-\omega_-= db ,
  \end{equation*}
  where $\omega_\pm = gI_\pm$.  Informally, it is a generalized
  K\"ahler structure where $g$ may be degenerate.
\end{defn}

Degenerate bi-Hermitian structures arising from the construction in
Theorem~\ref{brGK} as solutions to~\eqref{brane} enjoy a composition
operation which we now review (see~\cite{Gualtieri:2007fe} for
details).

If $F_{01}$ is a closed 2-form solving
\begin{equation}\label{ij}
  F_{01}I_0 + I_0^* F_{01} + F_{01}QF_{01} = 0,
\end{equation}
for a holomorphic Poisson structure $(I_0, \sigma_0)$ with
$\mathrm{Re}(\sigma_0)=Q$, then it determines a second holomorphic
Poisson structure $(I_1,\sigma_1)$ with $\mathrm{Re}(\sigma_1)=Q$, via
$I_1 = I_0 + QF_{01}$.  If we then have another closed 2-form
$F_{12}$, such that
\begin{equation}\label{jk}
  F_{12}I_1 + I_1^* F_{12} + F_{12}QF_{12} = 0,
\end{equation}
then it determines a third holomorphic Poisson structure
$(I_2,\sigma_2)$ with $\mathrm{Re}(\sigma_2)=Q$, via $I_2 = I_1 +
QF_{12}$.  Rewriting~\eqref{ij} and~\eqref{jk} as the pair
\begin{equation*}
  F_{01}I_0+I_1^* F_{01} =0,\hspace{6em}F_{12}I_1  +I_2^* F_{12} = 0,
\end{equation*}
we see that the closed 2-form $F_{02} = F_{01} + F_{12}$ satisfies
\begin{align*}
  F_{02}I_0 + I_0^*F_{02} + F_{02}QF_{02} &=F_{02}I_0  + I^*_2F_{02} \\
  &=F_{01}(I_2-I_1) - (I_1^* - I_0^*) F_{12}\\
  &= F_{01}QF_{12} - F_{01}QF_{12} =0.
\end{align*}
We may interpret this in the following way: a solution to~\eqref{ij}
defines a degenerate bi-Hermitian structure with constitutent complex
structures $(I_0, I_1)$, and a solution to~\eqref{jk} does the same,
but with complex structures $(I_1, I_2)$.  These two degenerate
bi-Hermitian structures may be composed in the sense that the sum
$F_{02}= F_{01} + F_{12}$ defines a new degenerate bi-Hermitian
structure with constituent complex structures $(I_0, I_2)$.  This
composition may be viewed as a groupoid (see Figure~\ref{grpd}).
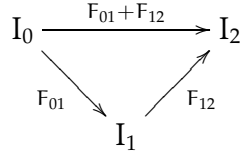
\begin{figure}
  \begin{equation*}
    \xymatrix{I_0\ar[dr]_{F_{01}}\ar[rr]^{F_{01}+F_{12}} & & I_2\\
      &I_1\ar[ru]_{F_{12}} & }
  \end{equation*}\caption{The composition of solutions to~\eqref{ij}, \eqref{jk}.}
  \label{grpd}
\end{figure}
\begin{defn}[\cite{Gualtieri:2007fe}]\label{groupoid} Fix a real
  manifold $M$ with real Poisson structure $Q$.  Then we may define a
  groupoid whose objects are holomorphic Poisson structures $(I_i,
  \sigma_i)$ on $M$ with $\mathrm{Re}(\sigma_i)=Q$ and whose morphisms
  $\Hom(i,j)$ are real closed 2-forms $F_{ij}$ such that the following
  two equations hold.
  \begin{align*}
    I_j - I_i = QF_{ij}\\
    F_{ij} I_j+ I_i^* F_{ij} = 0.
  \end{align*}
  The composition of morphisms is then simply addition of 2-forms
  $F_{ij} + F_{jk}$.
\end{defn}
\begin{remark}
  Combined with Theorem~\ref{brGK}, this definition provides a
  composition operation for the degenerate bi-Hermitian structures
  determined by the 2-forms $F_{ij}$.
\end{remark}

\section{Flow construction}\label{flowsec}
We now review a method, introduced in~\cite{MR2371181} and developed
in~\cite{Gualtieri:2007fe}, for modifying a bi-Hermitian structure of the kind studied in the previous section using
a smooth real-valued function.  The method proceeds essentially by
solving~\eqref{jk} using the flow of a suitably-chosen vector field,
and then composing this solution with the given bi-Hermitian structure
viewed as a solution to~\eqref{ij}.  This is a direct analog of the
well-known modification of a K\"ahler form by adding $f$ to the
K\"ahler potential.

\begin{theorem}[\cite{MR2371181,Gualtieri:2007fe}]\label{flow}
  Let $(I_0,\sigma_0)$ be a holomorphic Poisson structure with
  $Q=\mathrm{Re}(\sigma_0)$, and let $f$ be a smooth real-valued
  function.  Let $\varphi_t$ be the time-$t$ flow of the Hamiltonian
  vector field $X = Q(df)$.  Then, so far as the flow is well-defined,
  the closed 2-form
  \begin{equation}\label{efte}
    F_t = \int_0^t \varphi_s^*(dd^c_{I_0} f) ds
  \end{equation}
  satisfies Equation~\ref{brane}, i.e.
  \begin{equation*}
    F_t I_0 + I_0^*F_t + F_t Q F_t = 0.
  \end{equation*}
\end{theorem}
\begin{remark} The above flow generates a family of integrable complex
  structures $I_t = I_0 + QF_t$, which are all equivalent, since
  $I_t=\varphi_t(I_0)$. If $f$ is strictly plurisubharmonic for $I_0$,
  i.e. defines a Riemannian metric $h=-(dd^c_{I_0} f)I_0$, then
  from~\eqref{efte} we have
  \begin{equation*}
    \lim_{t\to 0} t^{-1}F_t = dd^c_{I_0} f,
  \end{equation*}
  implying that the symmetric tensor
  \begin{equation*}
    g_t = -\tfrac{1}{2}F_t(I_0 + I_t)
  \end{equation*}
  satisfies $\lim_{t\to 0} t^{-1}g_t = h$, so that $g_t$ defines a
  Riemannian metric for sufficiently small $t\neq 0$, and so by
  Theorem~\ref{flow}, we obtain a generalized K\"ahler structure
  $(g_t, I_0, I_t, b_t)$.
\end{remark}

\section{Generalized K\"ahler blow-up}\label{GKblp}

We now apply the machinery of the preceding sections to the problem of
blowing up the generalized K\"ahler 4-manifold $(M,\JJ_+,\JJ_-)$
introduced in Section~\ref{intro} at a nondegenerate point $p\in D_+$
in the complex locus of $\JJ_+$.  The first step (\S~\ref{gkblowup})
is to blow up the generalized complex structure $\JJ_+$ and obtain a
degenerate bi-Hermitian structure.  In the second step
(\S~\ref{gkblowup2}) we deform the degenerate bi-Hermitian structure
by composing it with another degenerate bi-Hermitian structure
obtained from the flow construction (\S~\ref{flowsec}).  Finally
(\S~\ref{pos}), we prove that the resulting deformation is
positive-definite, defining a generalized K\"ahler structure on the
blow-up.

\subsection{Simultaneous blow-up}\label{gkblowup}

\begin{lemma}\label{holnorm}
  In a neighbourhood of the nondegenerate point $p\in D_+$, there
  exist complex coordinates $(u_\pm, v_\pm)$ such that the holomorphic
  Poisson structure $(I_\pm,\sigma_\pm)$ is given by
  $u_\pm\del_{u_\pm}\wedge\del_{v_\pm}$.
\end{lemma}
\begin{proof}
  From the normal form for $\JJ_+$ near $p$ given by
  Equation~\ref{standard}, it follows that $P_+$ is isomorphic to
  $\mathrm{Im}(w\del_w\wedge\del_z)$.  In particular, $P_+$ vanishes
  linearly along $D_+$.  By Equations~\ref{cmmt}, \ref{ident}, and
  \ref{realpois}, and since $D_-$ is disjoint from $D_+$, it follows
  that $Q=-\tfrac{1}{2}[I_+, I_-] g^{-1}$ has linear vanishing along
  $D_-$ as well. This means that the holomorphic Poisson structure
  $\sigma_\pm$ is a section of a holomorphic line bundle $\wedge^2
  T_{1,0}$ with a nondegenerate zero at $p$.
  Hence we may choose $I_\pm$-complex coordinates $(u_\pm,v_\pm)$ near
  $p$ such that $\sigma_\pm = u_\pm\del_{u_\pm}\wedge\del_{v_\pm}$, as
  required.
\end{proof}

We now demonstrate that the coordinates $(u_\pm, v_\pm)$ placing
$\sigma_\pm$ into standard form are closely related to the coordinates
$(w,z)$ placing $\JJ_+$ into the standard form~\ref{standard}.

\begin{lemma}\label{wed}
  In a sufficiently small neighbourhood $U$ of $p$ where the following
  coordinates are defined, the functions $u_\pm, v_\pm$ lie in the
  ideal of $C^\infty(U,\CC)$ generated by $w,z$.
\end{lemma}
\begin{proof}
  Let $\rho_+$ be the generator~\eqref{standard} defined by $\JJ_+$ in
  $U$, and let $\rho_- = e^\beta$ be the generator defined by $\JJ_-$,
  which has symplectic type in $U$, so that $\beta = B + i\omega$ is a
  complex 2-form such that $\omega$ is symplectic.

  The holomorphic Poisson structures $\sigma_\pm=u_\pm
  \del_{u_\pm}\wedge \del_{v_\pm}$ define generalized complex
  structures in $U$ via the differential forms
  \[
  u_\pm + du_\pm\wedge dv_\pm \in e^{\sigma_\pm} \Omega^{2,0}_\pm.
  \]
  In~\cite{Gualtieri:2010fk}, it is shown that these holomorphic
  Poisson structures may be expressed as a certain ``wedge product''
  of the underlying generalized complex structures $(\JJ_+,\JJ_-)$.
  Explicitly, this provides the following identities\footnote{In
    general, if $\rho_\pm$ generate the canonical line bundles of
    $\JJ_\pm$, then $\rho_+^\top\wedge \rho_-$ generates
    $e^{\sigma_+}\Omega^{n,0}(M,I_+)$ and $\rho_+^\top\wedge
    \ol\rho_-$ generates $e^{\sigma_-}\Omega^{n,0}(M,I_-)$. Here $\rho^\top$ is the reversal anti-automorphism of forms. }:
  \begin{align*}
    e^{\ol\beta} (w - dw\wedge dz) &= e^{\lambda_-}(u_- + du_-\wedge dv_-)\\
    e^{ \beta} (w - dw\wedge dz) &= e^{\lambda_+}(u_+ + du_+\wedge
    dv_+),
  \end{align*}
  for smooth functions $\lambda_+, \lambda_-\in
  C^\infty(U,\CC)$. Comparing these equations to~\eqref{autexp}, we
  see that the argument in the proof of Theorem~\ref{indep} implies
  the required constraint on $u_\pm, v_\pm$.
\end{proof}

\begin{theorem}\label{liftcx}
  The complex structures $I_-, I_+$ underlying a generalized K\"ahler
  4-manifold $(M,\JJ_+,\JJ_-)$ both lift to the blow-up of $(M,
  \JJ_+)$ at a nondegenerate complex point $p\in D_+$ .
\end{theorem}
\begin{proof}
  Let $\psi:U\to\CC^2$ be the chart defined by $(w,z)$ in the normal
  form~\eqref{standard} and let $\varphi_\pm:U\to\CC^2$ be the chart
  defined by $(u_\pm,v_\pm)$ in the normal form given by
  Lemma~\ref{holnorm} .  Then $\chi_\pm=\psi\circ\varphi^{-1}_\pm$ is
  a diffeomorphism and $\chi_\pm(0)=0$.  The complex structure $I_\pm$
  lifts to the blow-up $\til M$ precisely when the diffeomorphism
  $\chi_\pm$ lifts to a diffeomorphism of blow-ups
  $\til\chi_\pm:[\varphi_\pm(U):0]\to [\psi(U):0]$.  This occurs if
  and only if $u_\pm$ and $v_\pm$ are contained in the ideal generated
  by $w, z$, which is itself guaranteed by Lemma~\ref{wed}.
\end{proof}
\begin{remark}
  It follows from the theorem that the complex structure $\til I_\pm$
  we obtain on the blow-up of $(M,\JJ_+)$ may be identified with the
  usual complex blow-up of $(M,I_\pm)$ at $p$. Furthermore, since the
  holomorphic Poisson structure $\sigma_\pm$ vanishes at $p$, it
  follows that $\sigma_\pm$ lifts to a holomorphic Poisson structure
  on the blow-up.
\end{remark}

We now apply Theorem~\ref{liftcx} to obtain a degenerate bi-Hermitian
structure on the blow-up of $(M,\JJ_+, \JJ_-)$ at $p\in D_+$.  Let
$(g,I_+, I_-, b)$ be the bi-Hermitian structure on $M$ defined by the
generalized K\"ahler structure.
\begin{corollary}\label{degfir}
  Let $(\til M, \til \JJ_+)$ be the blow-up of the generalized complex
  4-manifold $( M, \JJ_+)$ at the nondegenerate point $p\in D_+$, with
  blow-down map $\pi$.  Then $\til M$ inherits a degenerate
  bi-Hermitian structure $(\til g, \til b, \til I_+, \til I_-)$ such
  that $\pi:(\til M, \til I_\pm)\rightarrow (M,I_\pm)$ is a usual
  holomorphic blow-down and $\til g + \til b = \pi^*(g+b)$.
\end{corollary}

\subsection{Deformation of degenerate bi-Hermitian
  structure}\label{gkblowup2}

The degenerate bi-Hermitian structure on $\til M$ obtained in
Corollary~\ref{degfir} fails to define a generalized K\"ahler
structure because $\til g$ is not positive-definite along the
exceptional divisor $E$.  We now apply Theorem~\ref{flow} to
obtain\footnote{The flow construction may be applied equally well to
  degenerate bi-Hermitian structures.} a second degenerate
bi-Hermitian structure, which we use to modify $(\til g, \til b, \til
I_+, \til I_-)$.  The modification will leave the structures on $\til
M$ unchanged outside a tubular neighbourhood $V_E$ of $E$ which blows
down to a neighbourhood of $p$ in which $\JJ_-$ has symplectic type
and is given by a complex 2-form with imaginary part $\omega$.  Let
$\pi:\til M\to M$ denote the blow-down map, and write $\til \omega =
\pi^*\omega$ for the pull-back of the symplectic form to $V_E$.

First we describe the degenerate bi-Hermitian structure using the
formalism of Theorem~\ref{brGK}.  The complex structure $\til I_-$ and
the 2-form $\til\omega$ satisfy~\eqref{brane}, and so in $V_E$ we have
\begin{equation*}
  \til I_+ = \til I_- + \til Q\til \omega,
\end{equation*}
where $\til Q=\mathrm{Re}(\til \sigma_-)=\mathrm{Re}(\til \sigma_+)$,
as in~\eqref{cmmt}, and $\til\sigma_\pm$ is the blown up holomorphic
Poisson structure.  In the following, we construct a closed 2-form
$F_t$ in a possibly smaller tubular neighbourhood such that
\[
\til I_+^t = \til I_+ + \til QF_t
\]
defines a new complex structure $\til I^t_+$.  The final task,
completed in Section~\ref{pos}, will be to show that the
composition~\eqref{com}, in the sense of Definition~\ref{groupoid},
defines a generalized K\"ahler structure.
\begin{equation}\label{com}
  \xymatrix{\til I_-\ar[r]^{\til \omega}&\til I_+\ar[r]^{F_t}& \til I^t_+}
\end{equation}

We now construct $F_t$.  Let $(u,v)$ be $I_+$-holomorphic coordinates
near $p$ such that $\sigma_+ = u\del_u\wedge \del_v$, and let
$(u_0,v_0)= (u/v,v)$ and $(u_1,v_1)=(u,v/u)$ be the two affine charts
covering a tubular neighbourhood $V_E$ of the exceptional divisor
$E=u_1^{-1}(0)\cup v_0^{-1}(0)$.  Using $u_0, v_1$ as affine
coordinates on $E\cong\CC P^1$, we may describe the Fubini-Study
metric $\omega_E$ in terms of the K\"ahler potential
\[
f_0 = \log (\tfrac{u_0\ol u_0}{1 + u_0\ol u_0}) = \log
(\tfrac{1}{1+v_1\ol v_1}),
\]
which is smooth away from $u_0=0$ and satisfies $i\del\delbar f_0 =
\omega_E$.  Although $f_0$ is singular, we observe that its
Hamiltonian vector field is smooth:
\begin{align*}
  Q(df_0) &= \mathrm{Re}(u_0 \del_{u_0}\wedge \del_{v_0})d\log (\tfrac{u_0\ol u_0}{1 + u_0\ol u_0})\\
  &=\tfrac{1}{1+ u_0\ol u_0}\mathrm{Re}(\del_{v_0}).
\end{align*}
Hence $Q(df_0)$ defines a smooth Poisson vector field on $V_E$.

Now choose a bump function $\eps\in C^\infty(V_E,[0,1])$ which
vanishes on a smaller tubular neighbourhood $U_E\subset V_E$ and is
such that $1-\eps$ has compact support in a closed disc bundle $K$
over E, with $U_E\subset K\subset V_E$.  Consider the smooth function
$f_\eps\in C^\infty(V_E,\RR)$ given by
\[
f_\eps = \eps\log(u\ol u + v\ol v) = \eps\log(v_0\ol v_0(1 + u_0\ol
u_0)).
\]
Since $i\del\delbar\log(v_0\ol v_0(1+u_0\ol u_0)) =
i\del\delbar\log(1+u_0\ol u_0) = -i\del\delbar f_0$, it follows that
\begin{equation}\label{sinfun}
  f = c(f_0 + f_\eps),\ \ c\in \RR_{>0}
\end{equation}
has the property that $X = Q(df)$ is a smooth Poisson vector field in
$V_E$ and
\begin{equation}\label{cpctef}
  i\del\delbar f = \begin{cases}
    c\omega_E & \text{ in } U_E\\
    0 & \text{ outside } K
  \end{cases}
\end{equation}

For sufficiently small $\delta>0$, there exists an open neighbourhood
$V_E'$, with $K\subset V_E'\subset V_E$, on which the flow $\varphi_t$
of $X$ is well-defined for all $t\in(-\delta,\delta)$.  Also, choose
$\delta$ small enough so that there is a neighbourhood $V_E''$ with
$\ol{V_E''}\subset V_E'$, with $\varphi_t(K)\subset V_E''$ for
$t\in(-\delta,\delta)$.  Using~\eqref{cpctef}, we see that
$\varphi_t^*(i\del\delbar f)$ is smooth on $V_E'$, with compact
support contained in $V_E''$, for all $t\in(-\delta,\delta)$.

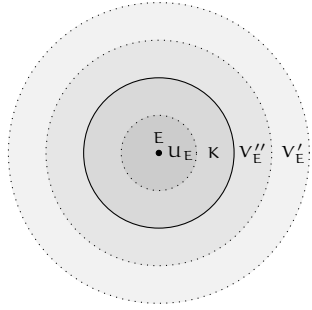
\begin{figure}\begin{center}\begin{tikzpicture}
      \filldraw [gray!10] (0,0) circle (2); \draw [dotted] (0,0)
      circle (2); \filldraw [gray!20] (0,0) circle (1.5); \draw
      [dotted] (0,0) circle (1.5); \filldraw [gray!30] (0,0) circle
      (1); \draw (0,0) circle (1); \filldraw [gray!40] (0,0) circle
      (.5); \draw [dotted] (0,0) circle (.5); \draw (0,.2) node
      {\tiny{$E$}}; \filldraw [black] (0,0) circle (1pt); \draw (.3,0)
      node {\tiny $U_E$}; \draw (.7,0) node {\tiny{ $K$}}; \draw
      (1.25,0) node {\tiny $V''_E$}; \draw (1.8,0) node {\tiny
        $V'_E$};
    \end{tikzpicture}\caption{Normal cross-section of neighbourhoods
      of $E$ contained in $V_E$}\end{center}\end{figure}

We now apply Theorem~\ref{flow}\footnote{The fact that $f$ is not
  smooth does not affect the validity of Theorem~\ref{flow} in this
  case, as the vector field $X=Q(df)$ is a smooth Poisson vector field, and hence locally Hamiltonian.} to the flow $\varphi_t$ on $V_E'$.
This provides a solution
\[
F_t = \int_0^t \varphi_s^*(dd^c_{\til I_+} f) ds
\]
to Equation~\ref{brane} for all $t\in(-\delta,\delta)$, with compact
support in $V'_E$.  Therefore, we obtain a family of complex
structures on $V_E'$ given by
\begin{equation}\label{defo}
  \til I_+^t = \til I_+ + Q F_t.
\end{equation}
Since $F_t$ has compact support contained in $V_E'$, the complex
structure $\til I_+^t$ may be extended to all of $\til M$ by setting
it equal to $\til I_+$ outside $V_E'$.  We summarize the above
procedure in the following result.
\begin{prop}
  The flow construction of Theorem~\ref{flow}, applied to the singular
  function $f$ given in~\eqref{sinfun}, produces a smooth family of
  solutions $(F_t)_{t\in(-\delta,\delta)}$ to~\eqref{brane} with
  compact support in a tubular neighbourhood of the exceptional
  divisor, and hence we obtain a degenerate bi-Hermitian structure
  \[
  (\til g'_t, \til b'_t, \til I_+^t, \til I_+)
  \]
  on $\til M$, where $\til I_+^t$ is given by~\eqref{defo} and $\til
  g'_t, \til b'_t$ are as in Theorem~\ref{brGK}, yielding
  \begin{equation}\label{defg}
    \til g'_t = -\tfrac{1}{2}F_t(\til I_+ + \til I_+^t).
  \end{equation}
\end{prop}
In Section~\ref{pos}, we compose the above degenerate bi-Hermitian
structure with that from Corollary~\ref{degfir} and show the resulting
structure is positive-definite.

\begin{remark}\label{zed}The family
  of complex structures $\til I_+^t$ on $\til M$ constructed above
  defines a deformation of the blow-up complex structure $\til I_+$ in
  the direction given by the class in $H^1( \mathcal{T})$ defined by
  the vector field $Z=Q(df)$, which is a holomorphic vector field on
  the annular neighbourhood of $E$ defined by $V_E\backslash K$.  The
  $(1,0)$ part of $Z$ in this annular neighbourhood is (in the
  $(u_0,v_0)$ chart)
  \begin{align*}
    Z^{1,0} &= c\til \sigma_+ ( d(\log(\tfrac{u_0\ol u_0}{1+u_0\ol u_0}) + \log(v_0\ol v_0(1+u_0\ol u_0)))\notag\\
    &= c (u_0\del_{u_0}\wedge\del_{v_0})(u_0^{-1}du_0 + v_0^{-1}dv_0)\notag\\
    &=c(\del_{v_0} - \tfrac{u_0}{v_0}\del_{u_0} ).
  \end{align*}
  This deformation class has a geometric interpretation: since $p\in
  D_+$ and $\sigma_+|_{D_+}=0$, the contraction
  \[
  \Tr(d\sigma_+|_{D_+})
  \]
  defines a holomorphic vector field $\chi$ on $D_+$.  The flow of
  $c\chi$ then provides a path $p(t)$ of points on $D_+$.  The family
  of blow-ups of $(M,I_+)$ at $p(t)$ provides a deformation of complex
  structure with derivative $[Z^{(1,0)}]$ at $t=0$.
\end{remark}

\subsection{Positivity}\label{pos} 
Now that we have constructed the two degenerate bi-Hermitian
structures on $\til M$ occurring in~\eqref{com}, we must argue that
their composition in the sense of Definition~\ref{groupoid} is
positive-definite.  The composition is the (a priori degenerate)
bi-Hermitian structure $(\til g_t, \til b_t, \til I_-, \til I_+^t)$,
where
\begin{align*}
  \til g_t &= -\tfrac{1}{2}(\til\omega + F_t)(\til I_- + \til I_+^t)\\
  \til b_t&= -\tfrac{1}{2}(\til \omega + F_t)(-\til I_- + \til I_+^t).
\end{align*}
Rewriting this, we obtain
\begin{align}\label{three}
  \til g_t &= -\tfrac{1}{2}\left(\til\omega(\til I_- + \til I_+)
    +\til\omega(\til I_+^t - \til I_+) +
    F_t(\til I_- - \til I_+) + F_t(\til I_+ + \til I_+^t)\right) \notag\\
  &=\til g + \til g'_t-\tfrac{1}{2}(\til\omega\til QF_t - F_t \til Q
  \til\omega),
\end{align}
where we use the fact that $\til I_+ - \til I_- = \til Q\til\omega$
and $\til I_+^t - \til I_+ = \til QF_t$.
\begin{theorem}\label{finalresult}
  Provided that $c$ in~\eqref{sinfun} is chosen small enough, the
  symmetric tensor $\til g_t$ defined by~\eqref{three} is
  positive-definite on $\til M$ for sufficiently small $t\neq 0$,
  defining a generalized K\"ahler structure on the blow-up.
\end{theorem}
\begin{proof}
  Since $F_t\to 0$ as $t\to 0$, it follows that $\til I_+^t\to \til
  I_+$ as $t\to 0$.  By Equation~\ref{defg}, therefore, we see that
  \[
  \lim_{t\to 0} \tfrac{1}{t} \til g'_t = - (dd^c_{\til I_+} f)(\til
  I_+) =\begin{cases}c\omega_E &
    \text{ in } U_E\\
    0 & \text{ outside } K
  \end{cases}
  \]
  where $\omega_E$ is the Fubini-Study metric.  This implies that
  $\til g_t'$ is positive-definite when restricted to $TE$ for
  sufficiently small nonzero $t$, and hence $\til g + \til g'_t$ is
  positive-definite in a neighbourhood of $E$ for sufficiently small
  nonzero $t$.  Also, the third summand in~\eqref{three} is
  proportional to $\tilde \omega$, which vanishes along $E$.

  Fix $c=c_0\in\RR_{>0}$ in the definition~\eqref{sinfun} of $f$, and
  let $U\subset U_E$ be a tubular neighbourhood of $E$ where the third
  summand in~\eqref{three} is so small that $\til g_t$ is
  positive-definite in $U$ for sufficiently small nonzero $t$.  Note
  that $\til g_t$ is certainly positive-definite outside $K$ (where it
  coincides with $\til g$), hence it remains to show that $\til g_t$
  is positive in the intermediate region $K\backslash U$.

  We have chosen $U$ so that the third term in~\eqref{three} is
  dominated there by the first two terms.  This means that at each
  point in $U$ and for each vector $v\neq 0$ (and for suficiently
  small nonzero $t$), we have
  \begin{align}\label{compa}
    |\til Q(F_tv,\til \omega v)| &< \til g(v,v) +\til g'_t(v,v)\notag\\
    &=\til g(v,v)-\tfrac{1}{2}F_t((\til I_+ + \til I_+^t)v,v) \notag\\
    &= \til g(v,v) - F_t(\til I_+v,v) -\tfrac{1}{2}\til Q(F_tv,F_tv)\notag \\
    &=\til g(v,v) - F_t(\til I_+v,v).
  \end{align}
  Since~\eqref{compa} holds for $c=c_0$, it will also hold in $U$ for
  $c = \lambda c_0$, for any $\lambda\in (0,1)$, since for $x,
  y\in\RR_{\geq 0}$ and $z\in\RR$, we have the implication
  \begin{align*}
    \left( x < y + z \right)\Rightarrow \left(\lambda x < \lambda
      ({y+z}) \leq y + \lambda z\right).
  \end{align*}
  Therefore we have shown positivity of $\til g_t$ in $U$ for any
  $0<c\leq c_0$, for sufficiently small nonzero $t$.

  Now observe that the first term of~\eqref{three}, i.e. $\til g$, is
  positive-definite on $K\backslash U$ and independent of $c$, whereas
  the second and third terms are each proportional to $c$.  Hence by
  choosing $c\neq 0$ sufficiently small, we ensure that $\til g_t$ is
  positive-definite on $K\backslash U$, in addition to $U$ and outside
  $K$, for sufficiently small $t\neq 0$.  This completes the proof.
\end{proof}

\section{Examples}

By the work of Goto~\cite{Goto:2009ix}, we know that the choice of a
holomorphic Poisson structure on a compact K\"ahler manifold gives
rise to a family of generalized K\"ahler structures deforming the
initial K\"ahler structure.  In this way, one obtains nontrivial
generalized K\"ahler structures on any compact K\"ahler surface with
effective anti-canonical divisor $D$.  Performing a K\"ahler blow-up
of such a surface at a point lying on $D$, we obtain a new K\"ahler
surface with effective anti-canonical divisor given by the proper
transform of $D$.  Hence we may apply the Goto deformation and obtain
a generalized K\"ahler structure on the blow-up.  We believe that our
construction gives an explicit realization of Goto's existence result
in this case, as evidenced by Remark~\ref{zed}.

In the non-algebraic case, or for noncompact surfaces, our
construction provides new generalized K\"ahler structures.  For
example, a result of Apostolov~\cite{MR1869698} states that for
surfaces with odd first Betti number, a bi-Hermitian structure which
is not strongly bi-Hermitian may only exist on blow-ups of minimal
class VII surfaces with curves.  If the minimal surface has a
generalized K\"ahler structure, therefore, we may employ our result to
obtain structures on the appropriate blow-ups.

\begin{example}[Diagonal Hopf surfaces]
  $X=S^3\times S^1$ admits a family of generalized K\"ahler structures
  with bi-Hermitian structure $(g, I_+,I_-)$ given by viewing $X$ as a
  Lie group, taking $g$ to be a bi-invariant metric, and $(I_+, I_-)$
  to be left and right-invariant complex structures compatible with
  $g$ (see~\cite{Gualtieri:2010fk} for details). In these examples,
  $D_+$ and $D_-$ are nonempty disjoint curves which sum to the
  anti-canonical divisor.  We may therefore blow up any number of
  points lying on $D_+\cup D_-$ and obtain generalized K\"ahler
  structures on these manifolds, which are diffeomorphic to
  $(S^3\times S^1)\cs k \ol{\CC P^2}$.  This provides another
  construction of bi-Hermitian structures on non-minimal Hopf
  surfaces, besides those discovered in~\cite{Pontecorvo, MR1096177}.
\end{example}

In a remarkable recent work~\cite{MR2719408}, Fujiki and Pontecorvo
obtained bi-Hermitian structures on hyperbolic and parabolic Inoue
surfaces as well as Hopf surfaces, by carefully studying the twistor
space of the underlying conformal 4-manifold. They then obtained
bi-Hermitian structures when these surfaces are properly blown up,
meaning that the surface is blown up at nodal singularities of the
anti-canonical divisor.  Finally, they obtained bi-Hermitian
structures on a family of deformations of such blowups.  We may of
course blow up their minimal examples at smooth points of the
anti-canonical divisor, using our procedure. It remains to determine
how the various bi-Hermitian structures now known on $(S^3\times
S^1)\cs k\ol{\CC P^2}$ are related.

\bibliographystyle{hyperamsplain} \bibliography{GCG}

\end{document}